\documentclass[11pt]{amsart} 

\usepackage{mathrsfs,cancel}          
\usepackage{bm,amsmath,amsthm}     
\usepackage{amssymb,empheq}            
\usepackage{euscript}           
\usepackage{wasysym,eufrak}
\usepackage[normalem]{ulem}
\usepackage{graphicx,enumerate,calc,lscape,color}
\definecolor{red}{rgb}{1,0,0}
\usepackage[matrix,arrow,curve,frame]{xy}    
\usepackage{bbm}
\usepackage[colorlinks]{hyperref}
\usepackage{tikz} 
\usetikzlibrary{shapes,arrows} 
\usetikzlibrary{cd}
\usepackage{tikz-cd}

\xymatrixcolsep{1.9pc}                          
\xymatrixrowsep{1.9pc}
\newdir{ >}{{}*!/-5pt/\dir{>}}                  

\newtheorem{prop}[subsection]{Proposition}

\newtheorem{lemma}[subsection]{Lemma}

\newtheorem{thm2}{Theorem}

\theoremstyle{definition}  
\newtheorem{define}[subsection]{Definition}
\newtheorem{example}[subsection]{Example}

\newtheorem{remark}[subsection]{Remark}

  {\end{list}}

\newcommand{\field}[1]  {\mathbb #1} 
\newcommand{\A}         {\field A}
\newcommand{\R}         {\field R}

\newcommand{\Z}         {\field Z}
\newcommand{\C}         {\field C}

\newcommand{\Q}         {\field Q}

\renewcommand{\P}       {\field P}

\DeclareMathOperator{\Spec}{Spec}

\DeclareMathOperator{\Hom}{Hom}

\DeclareMathOperator{\shE}{\mathcal{E}}
\DeclareMathOperator{\shF}{\mathcal{F}}
\DeclareMathOperator{\shG}{\mathcal{G}}

\DeclareMathOperator{\shO}{\mathcal{O}}

\newcommand{\period}    {{\makebox[0pt][l]{\hspace{2pt} .}}}
\newcommand{\comma}     {{\makebox[0pt][l]{\hspace{2pt} ,}}}

\newcommand{\F}{\mathbb{F}}

\numberwithin{equation}{subsection}

\usepackage{tikz} 

\begin{document}

\title[Sheaves on K3 surfaces and Galois representations]{Moduli spaces of sheaves on K3 surfaces and Galois representations}

\author{Sarah Frei}

\maketitle

\begin{abstract}
We consider two K3 surfaces defined over an arbitrary field, together with a smooth proper moduli space of stable sheaves on each. When the moduli spaces have the same dimension, we prove that if the \'etale cohomology groups (with $\Q_\ell$ coefficients) of the two surfaces are isomorphic as Galois representations, then the same is true of the two moduli spaces. In particular, if the field of definition is finite and the K3 surfaces have equal zeta functions, then so do the moduli spaces, even when the moduli spaces are not birational.
\end{abstract}

\section*{Introduction}
Given a K3 surface $S$ defined over an arbitrary field $k$, we can study moduli spaces $M$ of stable sheaves on $S$ with fixed Chern classes. Under mild conditions on the Chern classes, each such moduli space is a smooth, projective, geometrically irreducible variety with a natural symplectic structure. The best-studied example of such a moduli space is the Hilbert scheme of points, $S^{[n]}$, parameterizing zero-dimensional subschemes of length $n$ in $S$. These spaces have been well-studied over $\C$ because they are one of the few known families of compact hyperk\"ahler manifolds. It is a well-known result due to Huybrechts \cite{Huy}, O'Grady \cite{OG97} and Yoshioka \cite{Y}, recently summarized in \cite{PR}, that when $k=\C$ such a moduli space $M$ is actually deformation equivalent to $S^{[n]}$ for $n=\frac{1}{2}\dim M$, and this result was recently generalized to arbitrary fields by Charles in his proof of the Tate conjecture for K3 surfaces over finite fields \cite{C}. However, these moduli spaces are typically not birational to the Hilbert scheme.

For a projective variety $X$ defined over a finite field, let $Z(X,t)$ denote the zeta function of $X$. We prove here that the zeta function of $M$ is determined by the zeta function of $S$.

\begin{thm2}
\label{zeta}
Let $S_1$ and $S_2$ be K3 surfaces defined over a finite field such that $Z(S_1, t)=Z(S_2, t)$. Let $M_1$ and $M_2$ be smooth proper moduli spaces of stable sheaves on $S_1$ and $S_2$, respectively, with $\dim {M}_1=\dim {M}_2$. Then $Z({M}_1,t)=Z(M_2,t)$.
\end{thm2}

Since any two such moduli spaces need not be birational, the equality in Theorem \ref{zeta} is surprising. In particular, there need not be a geometric map between the moduli spaces that realizes this equality in point-counts over finite fields.

Consider the case where $S_1=S_2$. When the moduli space ${M}$ is fine and two-dimensional, ${M}$ is a K3 surface derived equivalent to the original K3 surface. In this case, our result about zeta functions for two moduli spaces on a fixed K3 surface was already proved by Lieblich and Olsson \cite[Thm.\ 1.2]{LO} and independently by Huybrechts \cite[Prop.\ 16.4.6]{HuyK3}. We extend their result to also hold when ${M}$ is not a fine moduli space. Their work was also generalized by Honigs \cite{Ho} to hold for any derived equivalent surfaces. In higher dimensions, it is an open question whether any two moduli spaces corresponding to a given K3 surface, under possible conditions on Chern classes, are derived equivalent once their dimensions coincide. If we speculate for a moment that they are \cite[Ch.\ 10 Questions and open problems]{HuyK3}, then our result is consistent with Orlov's conjecture that derived equivalent smooth, projective varieties have isomorphic motives with rational coefficients \cite[Conj.\ 1]{Or}. In particular, this conjecture would imply that derived equivalent smooth, projective varieties over a finite field have equal zeta functions. 
On the other hand, if we suppose instead that there are two such moduli spaces of the same dimension which are not derived equivalent, our result suggests that for this family of varieties, the zeta function is a very coarse invariant.

By the Lefschetz trace formula, the zeta function is determined by the action of the Frobenius endomorphism on the cohomology ring. Thus we will deduce Theorem \ref{zeta} from the following more general statement. Let $\ell$ be a prime different from the characteristic of $k$, and for any of the varieties $X$ below, let $\overline{X}=X\times_k\overline{k}$ where $\overline{k}$ is the algebraic closure of $k$.

\begin{thm2}
\label{mainthm}
Let $S_1$ and $S_2$ be K3 surfaces defined over an arbitrary field $k$ such that $H^2_{\acute{e}t}(\overline{S}_1, \Q_\ell)\cong H^2_{\acute{e}t}(\overline{S}_2,\Q_\ell)$ as $\mathrm{Gal}(\overline{k}/k)$-representations. Additionally, let ${M}_1$ and ${M}_2$ be smooth proper moduli spaces of stable sheaves on $S_1$ and $S_2$, respectively, with $\dim {M}_1=\dim {M}_2$. Then for all $i\geq 0$, $H^{i}_{\acute{e}t}(\overline{{M}}_1, \Q_{\ell})\cong H^{i}_{\acute{e}t}(\overline{{M}}_2, \Q_{\ell})$ as $\mathrm{Gal}(\overline{k}/k)$-representations.
\end{thm2}

We remark that when the moduli spaces are fine, the isomorphism $H^{2}_{\acute{e}t}(\overline{{M}}_1, \Q_{\ell})\cong H^{2}_{\acute{e}t}(\overline{{M}}_2, \Q_{\ell})$ follows almost immediately from the work of Charles \cite{C}, who built off of work done by O'Grady \cite{OG97} over the complex numbers. We extend their results to non-fine moduli spaces, and then the bulk of the work required to prove Theorem \ref{mainthm} is to construct the Galois-equivariant isomorphisms for the higher cohomology groups. 

\subsection*{Outline} In Section \ref{lattice} we introduce notation and definitions in order to define the moduli space $M$ of stable sheaves on a K3 surface $S$, and we show that it is a smooth, projective, geometrically irreducible variety. We show in Section \ref{orthog} that $H^2(\overline{{M}}, \Z_\ell(1))$ is isometric to a specific sublattice in $H^*(\overline{S}, \Z_\ell)$ and in Section \ref{i=1} that, after tensoring with $\Q_\ell$, the same sublattice can be identified with a fixed sublattice of $H^*(\overline{S}, \Q_\ell)$, which depends only on the dimension of $M$. In Section \ref{single}, we reduce to the case of considering just one K3 surface $S$ and comparing ${M}$ to the Hilbert scheme $S^{[n]}$. In Section \ref{theiso}, we complete the proof of Theorem \ref{mainthm} by constructing a Galois equivariant ring isomorphism between the cohomologies of the two moduli spaces ${M}$ and $S^{[n]}$.

\subsection*{Acknowledgments} I thank my advisor, Nicolas Addington, for supporting me and teaching me the material necessary to complete this project. Thanks also to Adrian Langer, Fran\c{c}ois Charles, Max Lieblich, Eduard Looijenga, Valery Lunts, Mark de Cataldo, Luc Illusie, Andy Putman, and Katrina Honigs for the helpful discussions and correspondences, and the referee who suggested many improvements.

\section{The moduli space}
\label{lattice}

Let $S$ be a K3 surface defined over an arbitrary field $k$ with algebraic closure $\overline{k}$, and let $\overline{S}=S\times_k \overline{k}$. The notion of a Mukai lattice makes sense in any Weil cohomology theory and is discussed in great generality in \cite{LO} as well as \cite{Ho}. We will make use of Mukai's original construction over $\C$ as well as the construction for \'etale cohomology, which we recall here.

\begin{define}
\label{latticedef}
Let $\ell$ be a prime different from the characteristic of $k$. The \emph{$\ell$-adic Mukai lattice} of $S$ is the $\mathrm{Gal}(\overline{k}/k)$-module
$$\widetilde{H}(\overline{S}, \Z_\ell):= H^0(\overline{S}, \Z_\ell)\oplus H^2(\overline{S}, \Z_\ell(1)) \oplus H^4(\overline{S}, \Z_\ell(2))$$
endowed with the Mukai pairing
$$(\alpha, \beta) = -\alpha_0.\beta_4 + \alpha_2.\beta_2-\alpha_4.\beta_0.$$
\end{define}

Note that we have defined the Mukai lattice in weight zero but will continue to use the usual sign on the Mukai pairing.

\begin{define}
For a coherent sheaf $\shF$ on $S$, the \emph{Mukai vector of $\shF$} is
$$v(\shF):= \text{ch}(\shF) \sqrt{\text{td}(S)} =(\text{rk}\shF, c_1(\shF), \chi(\shF)-\text{rk}\shF)\in \widetilde{H}(\overline{S}, \Z_\ell).$$
\end{define}

For two coherent sheaves $\shF$ and $\shG$, the Euler pairing and the Mukai pairing are related by $\chi(\shF, \shG)=-(v(\shF), v(\shG))$. Mukai vectors of sheaves are elements of the following subgroup of $\widetilde{H}(\overline{S}, \Z_\ell)$.

\begin{define}
Let $\omega \in H^4(\overline{S}, \Z_\ell(2))$ be the numerical equivalence class of a point on $\overline{S}$. A \emph{Mukai vector on $S$} is an element of
$$N(S):=\Z\oplus \mathrm{NS}(S)\oplus \Z\omega,$$ 
and $N(S)$ is considered as a subgroup of $\widetilde{H}(\overline{S}, \Z_\ell)$ under the natural inclusion. A Mukai vector is often denoted by $v=(r, c_1, s)$.
\end{define}

Given a Mukai vector $v$ on $S$ and an ample class $H$ in NS$(S)$, 
we can form the moduli space ${M}_H(S, v)$ of Gieseker geometrically $H$-stable sheaves $\mathcal{F}$ on $S$ such that $v(\mathcal{F})=v$. These spaces were originally constructed over algebraically closed fields in \cite{Maruyama} and \cite{Gieseker}. When the notation is clear, we will simply write ${M}$ or $M(v)$ in place of ${M}_H(S, v)$. By \cite[Thm.\ 0.2]{La}, ${M}$ is a quasi-projective scheme of finite type over $k$. In order for the moduli space to be a non-empty, smooth, projective variety, we will require the Mukai vector to satisfy the following conditions.

\begin{define}
\label{geoprim}
A Mukai vector $v \in N(S)$ is \emph{geometrically primitive} if its image under $N(S) \to N(\overline{S})$ is primitive.
\end{define}

Geometrically primitive is the same as primitive when Br$(k)=0$, or when $S$ has a zero-cycle of degree one (for example, a $k$-point), in which case there is an isomorphism Pic$(S) \xrightarrow{\sim} \text{Pic}(\overline{S})^{\mathrm{Gal}(\overline{k}/k)}$ coming from the Hochschild-Serre spectral sequence \cite[Eq.\ 18.1.10 and 18.1.13]{HuyK3}.

\begin{define}
\label{effective}
A Mukai vector $v=(r,c_1,s)\in N(S)$ is \emph{effective} if $r>0$, or if $r=0$ and $c_1$ is effective, or if $r=c_1=0$ and $s>0$.
\end{define}

These conditions are necessary to ensure that $M(v)$ is non-empty.

\begin{define}
\label{generic}
If $v$ is geometrically primitive, a polarization $H\in \mathrm{Pic}(S)$ is \emph{v-generic} if every $H$-semistable sheaf is geometrically stable with respect to $H$.
\end{define}

On $\overline{S}$, if $r>0$ or $r=0$ and $s\neq 0$, there are many choices of $v$-generic polarizations. Any polarization which is not contained in the locally finite union of hyperplanes in NS$(\overline{S})_{\R}$ discussed in \cite[App.\ 4.C]{HuyLehn} is $v$-generic, by \cite[Thm.\ 4.C.3]{HuyLehn} for $r>0$ and \cite[Sec.\ 1.4]{Y} for $r=0$ and $s\neq 0$. However, it is possible that NS$(S)\subset \mathrm{NS}(\overline{S})$ is contained entirely in one of the walls in NS$(\overline{S})_{\R}$, resulting in the existence of properly semistable sheaves. Here we give an explicit example of a K3 surface $S$ and a Mukai vector $v$ with $r>0$ for which no $v$-generic polarization exists. 

\begin{example}
Let $S$ be the degree two K3 surface defined over $\F_3$ cut out by 
\begin{align*}
w^2=&2y^2(x^2+2xy+2y^2)^2+(2x+z)(x^5+x^4y+x^3yz+x^2y^3+x^2y^2z \\
	&+2x^2z^3+xy^4+2xy^3z+xy^2z^2+y^5+2y^4z+2y^3z^2+2z^5)
\end{align*}
in $\P(3,1,1,1)$ which is given in \cite[Section 5]{HassVA}. In the proof of Proposition $5.5$ in \cite{HassVA}, Hassett and V\'arilly-Alvarado show that $\mathrm{NS}(S)=\Z H$, where $H$ is the preimage of a hyperplane from $\P^2$, and that $\mathrm{NS}(\overline{S})=\Z C_1 \oplus \Z C_2$, where $C_1$ and $C_2$ are $-2$-curves defined over $\F_9$ which lie above the tritangent line $2x+z=0$ in $\P^2$. The inclusion $\mathrm{NS}(S)\subset \mathrm{NS}(\overline{S})$ is given by $H=C_1+C_2$. We claim that there is no $v$-generic polarization on $S$ for $v=(2,-H,0)$. 

Let $\shE$ be a non-trivial extension of $\shO_S(-C_1)$ by $\shO_S(-C_2)$, which exists since $\mathrm{Ext}^1(\shO_S(-C_1),\shO_S(-C_2))=\F_3^3$. It can be checked that $v(\shE)=(2,-H,0)$, and that the reduced Hilbert polynomials with respect to $H$ satisfy $p_{\shO_S(-C_2)}(t)=p_{\shE}(t)=t^2-t+1$. Furthermore, every other saturated subsheaf of $\shE$ has strictly smaller reduced Hilbert polynomial, so $\shO_S(-C_2)$ is the only destabilizing subsheaf. Thus the locus of geometrically stable sheaves will be a proper subvariety of the moduli space of semistable sheaves on $S$ with $v=(2,-H,0)$.

This example further demonstrates that $M_H(v)$ need not exist when $H$ is not defined over the ground field. Consider $H'=H+\epsilon C_1$ for small $\epsilon>0$. We claim that $M_{H'}(v):=M_{H'}(S_{\F_9}, v)$ cannot be constructed over $\F_3$. If it could, then the $\mathrm{Gal}(\overline{\F}_3/\F_3)$-action would permute the points of $M_{H'}(v)_{\overline{\F}_3}$. However, the sheaf $\shE$ from above is in $M_{H'}(v)$, and its image under the action of the Frobenius is not, giving a contradiction. We see this by looking at the eigenvalues of the Frobenius action $\sigma$ given in the proof of \cite[Prop.\ 5.5]{HassVA}, which shows that $\sigma$ on $\overline{S}$ swaps $C_1$ and $C_2$, so applying $\sigma^*$ to the sequence
$$0\to \mathcal{O}_S(-C_2) \to \shE \to \mathcal{O}_{S}(-C_1) \to 0$$
swaps $\mathcal{O}_S(-C_2)$ and $\mathcal{O}_{S}(-C_1)$. This shows that $\mathcal{O}_{S}(-C_2)$ is a destabilizing quotient of $\sigma^*\shE$.
\end{example}

In light of the example above, we will restrict ourselves to situations in which properly semistable sheaves do not exist. This can be guaranteed, for example, if $r>0$ and the components of $v$ satisfy a gcd condition, as Charles assumes in \cite[Def.\ 2.3]{C}, or if $\text{rank}(\mathrm{NS}(S))=\text{rank}(\mathrm{NS}(\overline{S}))$, as Huybrechts assumes in \cite{HuyMotives}, and $r>0$ or $r=0$ and $s\neq 0$. For simplicity, we will always pick polarizations satisfying Definition \ref{generic}.

\begin{prop}
\label{moduli}
Let $v\in N(S)$ be an effective and geometrically primitive Mukai vector with $v^2\geq 0$, and let $H$ be a $v$-generic polarization on $S$. Then ${M}$ is a non-empty, smooth, projective, geometrically irreducible variety over $k$ of dimension $v^2+2$.
\end{prop}

This was proved in \cite[Thm.\ 2.4(i)]{C} under the stronger assumption that $v$ satisfy condition $(C)$ given in \cite[Def.\ 2.3]{C}, which in particular implies that ${M}$ is a fine moduli space. See also \cite[Prop.\ 4.5]{FuLi} for a similar result which is slightly more general than \cite[Thm.\ 2.4(i)]{C}, but which still requires $v$ to have positive rank.

\begin{proof}[Proof of Proposition \ref{moduli}]
First, we show that ${M}$ is projective. By \cite[Thm.\ 0.2]{La}, it is enough to show that any semistable sheaf in $M$ is actually geometrically stable. But this follows immediately since $v$ is geometrically primitive and $H$ is $v$-generic. 

For smoothness, we know $\overline{M}=M_{\overline{k}}$ is smooth by \cite[Cor.\ 0.2]{Mu}, and hence $M$ is also smooth. Once we know $M$ is non-empty, discussed below, \cite[Cor.\ 0.2]{Mu} also shows that $\dim M=v^2+2$.

We show next that ${M}$ is geometrically irreducible. Over $\C$, this fact is well-known: see \cite[Thm.\ 4.1]{KLS} or \cite[Thm.\ 8.1]{Y}. In fact, the proof of \cite[Thm.\ 4.1]{KLS} works over any algebraically closed field, and so $\overline{M}$ is irreducible when $M$ is defined over any field.

Lastly, the non-emptiness of $M$ over $k=\C$ is due to Mukai and Yoshioka and is summarized in \cite[Thm.\ 10.2.7]{HuyK3}. 
For any other $k$ with $\mathrm{char}\,k=0$, we apply the Lefschetz principle. 
For $\mathrm{char}\,k>0$, we will lift $\overline{S}$ to characteristic zero. By \cite[Prop.\ 1.5]{C}, there is a finite flat morphism $\Spec W' \to \Spec W$, where $W$ is the ring of Witt vectors of $\overline{k}$ and $W'$ is a discrete valuation ring with fraction field of characteristic zero and residue field $\overline{k}$, and a smooth projective relative K3 surface $\mathcal{S} \to \Spec W'$ with special fiber isomorphic to $\overline{S}$. Moreover, there are lifts $\mathcal{H}$ of $H$ and $\tilde{c}_1$ of $c_1$ to $\mathcal{S}$, so we can form the relative moduli space $f\colon \mathcal{M}_{\mathcal{H}}(\mathcal{S}, v_{W'})\to \Spec W'$ parameterizing geometrically stable sheaves on the fibers of $\mathcal{S}\to\Spec W'$, which is projective by \cite[Thm.\ 0.2]{La}. In particular, $f$ is proper and the generic fiber is defined over a field of characteristic zero. Thus, the generic fiber is non-empty by the argument above, so the special fiber $\overline{M}$ is non-empty, so $M$ is non-empty.
\end{proof}

\section{Generalizing results of Mukai and O'Grady}
\label{orthog}

In \cite{Mu2}, Mukai showed that for a complex projective K3 surface $S$ and a primitive Mukai vector $v$ with $v^2=0$, there is a Hodge isometry $v^\perp/\langle v\rangle \cong H^2_{sing}({M}, \Z)$, where $v^\perp$ is the orthogonal complement of $v$ in the Mukai lattice, $M=M(v)$, and the pairing on $H^2_{sing}({M}, \Z)$ is given by the Beauville-Bogomolov form. Similarly, when $v^2>0$, O'Grady \cite{OG97} proved that $v^\perp \cong H^2_{sing}({M}, \Z)$. In \cite[Def.\ 26.19]{GHJ}, Huybrechts shows that, up to scaling, we can define the Beaville-Bogomolov form on a hyperk\"ahler variety $X$ by
$$\tilde{q}_X(\alpha)=\int_X \alpha^2\sqrt{\text{td}\,X},$$
and now this definition makes sense for \'etale cohomology as well. 

We will prove here that the isometries proved by Mukai and O'Grady also hold in \'etale cohomology when $S$ is defined over an arbitrary field $k$ and $v$ is an effective and geometrically primitive Mukai vector.

\begin{prop}
\label{H^2iso}
Let $S$ be a K3 surface defined over an arbitrary field $k$ and $v\in N(S)$ an effective and geometrically primitive Mukai vector with a $v$-generic polarization $H\in \mathrm{NS}(S)$.
\begin{enumerate}[(i)]
\item When $v^2>0,$ there is a Galois equivariant isometry $$v^\perp \cong H^2(\overline{{M}}, \Z_\ell(1)).$$
\item When $v^2=0,$ there is a Galois equivariant isometry $$v^\perp/\langle v \rangle \cong H^2(\overline{{M}}, \Z_\ell(1)).$$
\end{enumerate}
\end{prop}

Charles in \cite[Thm.\ 2.4(v)]{C} proved this result when $v^2>0$ with stronger assumptions on $v$. We follow his technique to prove the more general result, making modifications where necessary. We will prove Proposition \ref{H^2iso} in great detail so that we can easily refer back to it in a similar situation later.

\begin{proof}[Proof of Proposition \ref{H^2iso}]
In order to define the Mukai map which will give the desired isomorphisms, we must first show that a quasi-universal sheaf $\mathcal{U}$ exists on $S\times {M}$, in the sense of \cite[Def.\ 4.6.1]{HuyLehn}. We claim that the proof of existence given in \cite[Prop.\ 4.6.2]{HuyLehn} holds over any field by appealing to work by Langer on moduli of sheaves in mixed characteristic. Langer proves in \cite[Thm.\ 4.3]{LaDuke} that the quotient $\mathscr{R}^s \to M$, where $\mathscr{R}^s$ is the subset of the Quot-scheme parameterizing stable sheaves, is a principal PGL$(V)$-bundle in the fppf topology and by \cite[I.3.26]{Mi}, it also has local sections in the \'etale topology. Then the proof of \cite[Prop.\ 4.6.2]{HuyLehn} gives that the universal sheaf on $S\times \mathscr{R}^s$ descends to a quasi-universal sheaf on $S\times {M}$.

Introducing some notation, we let $\pi_1$ and $\pi_2$ be the two projections from $S\times {M}$:
\[\xymatrix{  &  S\times {M} \ar[dl]_{\pi_1} \ar[dr]^{\pi_2}&  \\
			S & 			&  {M}. }
\]
The Mukai map $\theta_v\colon \widetilde{H}(\overline{S},\Z_\ell) \to H^2(\overline{{M}}, \Z_\ell(1))$ is defined by
$$\alpha \mapsto \frac{1}{\rho}\left[\pi_{2*}(v(\mathcal{U})\cdot \pi_1^*(\alpha))\right]_{2},$$
where $v(\mathcal{U})$ is the Mukai vector of $\mathcal{U}$ and $\rho$ is the similitude of $\mathcal{U}$ (that is, the rank of the sheaf $W$ in \cite[Def.\ 4.6.1]{HuyLehn}). 

We are now ready to prove that if $v^2>0$, then $\theta_v|_{v^\perp}$ gives $(i)$. This will be done in different cases depending on the field $k$. If $k=\C$, then $\theta_v$ was proved in \cite[Main Thm.]{OG97} to be an isometry for singular cohomology with coefficients in $\Z$. This isomorphism can be tensored with $\Z_\ell(1)$, and then the comparison theorem for singular and \'etale cohomology gives the isomorphism $v^\perp \cong H^2(M, \Z_\ell(1))$. 

Now suppose $k$ is an arbitrary field of characteristic zero. Again by the Lefschetz principle, there is a field $k'$ with inclusions $\overline{k'}\hookrightarrow \C$ and $\overline{k'}\hookrightarrow \overline{k}$ such that $S$ and $M$ are defined over $k'$. The inclusions give the following horizontal isomorphisms by smooth base change:
\[\xymatrix{ H^2(\overline{M}, \Z_\ell(1)) & \ar[l]_{\sim} H^2(M_{\overline{k'}}, \Z_\ell(1)) \ar[r]^{\sim} &  H^2(M_\C, \Z_\ell(1)) \\
\ar[u]^{\theta_v} v^\perp_{\overline{k}} & \ar[u] \ar[l]_{\sim} v^\perp_{\overline{k'}} \ar[r]^{\sim} & \ar[u]_{\theta_v}^-{\rotatebox{90}{$\sim$}}  v^\perp_\C \comma
}\] 
where $\displaystyle v^\perp_{\overline{k'}}\subseteq \widetilde{H}(S_{\overline{k'}}, \Z_\ell)$, and similarly for $v^\perp_{\overline{k}}$ and $v^\perp_\C$. The right-most vertical arrow is an isomorphism by the argument above, and by commutativity this implies the other vertical arrows are isomorphisms as well. 

Next, suppose $k$ is an arbitrary field of characteristic $p>0$. As in the proof of Proposition \ref{moduli}, we form the relative moduli space \linebreak$\mathcal{M}=\mathcal{M}_{\mathcal{H}}(\mathcal{S},v_{W'})$, a projective scheme over $\Spec W'$ whose central fiber is $\overline{M}$. We will use smooth base change below, so we first show that $f\colon \mathcal{M}\to \Spec W'$ is a smooth morphism. This will follow by showing that $f$ is smooth at closed points in the central fiber, which are the closed points of $\overline{M}$, so they correspond to geometrically stable sheaves $\shF$ on $\overline{S}$. By \cite[Lem.\ 3.1.5]{HLT}, $f$ is smooth at such a point $[\shF]$ if and only if Pic$(\mathcal{S}/W')$ is smooth at $[\det\shF]$. The latter is smooth because $\det \shF=c_1$ lifted from $\overline{S}$ to $\mathcal{S}$.

By the same argument given above, there is a quasi-universal sheaf $\,\widetilde{\mathcal{U}}$ on $\mathcal{S}\times_{W'} \mathcal{M}$. Continuing to use $\pi_1$ and $\pi_2$ for the projections to $\mathcal{S}$ and $\mathcal{M}$, we can define a relative Mukai map 
$$\widetilde{\theta}_v\colon v^\perp_{W'} \to H^2(\mathcal{M}, \Z_\ell(1)),$$ 
where $v^\perp_{W'}\subset \widetilde{H}(\mathcal{S}, \Z_\ell)$, and where $\widetilde{\theta}_v(\alpha)=\frac{1}{\rho}[\pi_{2*}(v(\widetilde{\mathcal{U}})\cdot \pi_1^*(\alpha))]_2$. We observe that $\widetilde{\theta}_v$ restricts exactly to the map $\theta_v$ over both fibers.

Next we conclude by  \cite[VI.4.2]{Mi} that the cohomology groups for all geometric fibers of $\mathcal{M}\to \Spec W'$ are isomorphic. In particular, if we let $K:=\text{Frac}\, W'$,  it follows that $H^2(M_{\overline{K}}, \Z_\ell(1))\cong H^2(\overline{M}, \Z_\ell(1))$. Similarly, the corresponding Mukai lattices are also isomorphic. Thus, the smooth base change theorem gives the following commutative diagram with horizontal isomorphisms, where the right-most vertical arrow is an isomorphism because $\mathrm{char}\,{K}=0$:
\[\xymatrix{ H^2(\overline{M},\Z_\ell(1)) & H^2(\mathcal{M},\Z_\ell(1)) \ar[l]_-{\sim} \ar[r]^-{\sim} & H^2(M_{\overline{K}},\Z_\ell(1)) \\
v^\perp_{\overline{k}} \ar[u]_-{\theta_v} & v^\perp_{W'} \ar[l]_-{\sim} \ar[r]^-{\sim} \ar[u]_-{\widetilde{\theta}_v} & v^\perp_{\overline{K}} \ar[u]_-{\theta_v}^-{\rotatebox{90}{$\sim$}} \period}\]
Therefore, the left-most vertical arrow is also an isomorphism, as desired. 

In all cases, $v(\mathcal{U})$ is defined over $k$, so $\theta_v$ is Galois equivariant, and it continues to respect the Mukai and Beauville-Bogomolov pairings as shown by \cite[Main Thm.]{OG97}. Hence $\theta_v$ is a Galois equivariant isometry. This completes the proof of $(i)$.

The proof of $(ii)$ follows the same argument, using the isometry\linebreak$v^\perp/\langle v \rangle \xrightarrow{\sim} H^2(M, \Z)$ proved in \cite[Thm.\ 1.4]{Mu2} for $k=\C$ in place of \cite[Main Thm.]{OG97}. 
\end{proof}

\section{A Galois equivariant isometry}
\label{i=1}

To prove Theorem \ref{mainthm} for $i=2$, it remains to show the following:

\begin{prop}
\label{perps}
Let $v\in N(S)$ be a non-zero Mukai vector on a K3 surface $S$ defined over an arbitrary field $k$, and consider $v^\perp\subset \widetilde{H}(\overline{S}, \Q_\ell)$. 
\begin{enumerate}[(i)]
\item When $v^2>0$, there is a Galois equivariant isometry $$v^\perp \cong H^2(\overline{S}, \Q_\ell(1))\oplus \Q_\ell,$$ where the pairing on the right side is given by the intersection form on $H^2(\overline{S}, \Q_\ell(1))$ and $-v^2$ on the generator of $\Q_\ell$.
\item When $v^2=0$, there is a Galois equivariant isometry $$v^\perp/\langle v \rangle \cong H^2(\overline{S}, \Q_\ell(1)),$$ where the pairing on the right side is given by the intersection form.
\end{enumerate}
\end{prop}

Note that Proposition \ref{perps} need not hold when $\Q_\ell$ is replaced with $\Z_\ell$, as demonstrated by the following example. 

\begin{example}
We consider the K3 surface $S$ defined over $\F_2$ in \cite[Ex.\ 6.1]{HVV}, which by the proof of \cite[Prop.\ 6.3]{HVV} has rank$(\mathrm{NS}(\overline{S}))= 2$. In particular, Hassett, V\'arilly-Alvarado, and Varilly find two independent classes in $\mathrm{NS}(S)$ on which the intersection pairing is
\[\left(\begin{array}{rr}-2 & 3 \\ 3 & -2 \end{array}\right).\]
Since this has a square-free discriminant of $-5$, the span of the two classes is a primitive sublattice in $\mathrm{NS}(S)$, and hence the classes span $\mathrm{NS}(S)$.

In this case, we can consider the effective and geometrically primitive Mukai vector $v=(5,2,3,0)\in N(S)$. Since $\text{rank}(\mathrm{NS}(S))=\text{rank}(\mathrm{NS}(\overline{S}))$, there is a $v$-generic polarization, and hence $M(v)$ is a 12-dimensional smooth projective variety. There is no $u\in N(S)$ such that $( u, v)=1$, so we do not expect $M(v)$ to be a fine moduli space. If there is an isometry $v^\perp\cong H^2(\overline{S},\Z_\ell(1))\oplus \Z_\ell$, then we can restrict it to the subspace of Galois invariants. The proof of \cite[Prop.\ 6.3]{HVV} also shows that the only invariant classes in $H^2(\overline{S},\Z_\ell(1))$ are those in $\mathrm{NS}({S})$, and so the sublattice $H^2(\overline{S},\Z_\ell(1))^{\mathrm{Gal(\overline{\F}_2/\F_2)}}\oplus \Z_\ell$ has discriminant 50. 
It can be checked that the pairing on $(v^\perp)^{\mathrm{Gal(\overline{\F}_2/\F_2)}}$ is
\[\left(\begin{array}{rrr}  -2 & 3 & -1 \\ 3 & -2 & 0 \\ -1 & 0 & 0 \end{array}\right),\]
which has discriminant 2. For these lattices to be isomorphic, the discriminants must differ by the square of a unit, but when $\ell=5$, this is not the case. So $v^\perp \not\cong H^2(\overline{S},\Z_5(1))\oplus\Z_5$ as sublattices of $\widetilde{H}(\overline{S},\Z_5)$. By Proposition \ref{perps}, it is only after tensoring with $\Q_5$ that these lattices become isomorphic.

This difference in coefficients is related to whether $M(v)$ is birational to the Hilbert scheme. If $w=(1,0,0,-5)$ in $N(S)$, then $M(w)=S^{[6]}$ and it is clear that $w^\perp=H^2(\overline{S}, \Z_\ell(1))\oplus \Z_\ell\langle (1,0,0,5)\rangle$. For hyperk\"ahler varieties defined over the complex numbers, the Beauville-Bogomolov form and hence the resulting discriminant group is a birational invariant. While this result has not been proved over arbitrary fields, our calculations suggest that we have found two moduli spaces that are not birational.
\end{example}

\begin{proof}[Proof of Proposition \ref{perps}]
To prove $(i)$, let $w:=(1,0,1-n)\in N(S)$ where $n=\frac{1}{2}(v^2+2)$. If $n>1$ then $w^\perp=H^2(\overline{S}, \Q_\ell(1))\oplus \Q_\ell\langle (1,0,n-1)\rangle$, so we will use reflections to prove that $v^\perp\cong w^\perp$. Observe that either $(v-w)^2\neq 0$ or $(v+w)^2\neq 0$. Then reflection through $v-w$ or $v+w$, respectively, gives a map $\widetilde{H}(\overline{S}, \Q_\ell) \to \widetilde{H}(\overline{S}, \Q_\ell)$. It can be checked that this reflection preserves the Mukai pairing, sends $v$ to $\pm w$, and induces a map $v^\perp \xrightarrow{\sim} w^\perp$ which is Galois equivariant. 
This completes the proof of $(i)$.

The proof of $(ii)$ requires a few modifications to the argument above. We now consider $w=(1,0,0)\in N(S)$, so that $w^\perp/\langle w\rangle =H^2(\overline{S}, \Q_\ell(1))$. If $(v-w)^2\neq 0$, then as above, reflection through $v-w$ restricts to a $\mathrm{Gal}(\overline{k}/k)$-equivariant isometry $v^\perp/\langle v \rangle \xrightarrow{\sim} w^\perp/\langle w \rangle$.

If instead $(v-w)^2=0$, then $(v,w)=0$ and $(v+w)^2=0$ as well. Let us write $v=(r,c_1,0)$. If $r\neq 0$, then reflecting through $v-(0,0,1)$ gives that $v^\perp/\langle v \rangle\cong (0,0,1)^\perp/\langle (0,0,1)\rangle$. Then $(0,0,1)^\perp/\langle (0,0,1)\rangle\cong w^\perp/\langle w\rangle$ by reflecting through $(0,0,1)-w$.

Lastly, suppose $v=(0,c_1,0)$. Since $c_1^2=0$ and $c_1 \neq 0$, by the Hodge Index Theorem, any ample divisor $h$ on $S$ is such that $c_1.h \neq 0$. If we let $v'=ve^h=(0,c_1,c_1.h)$, then it follows that $(v'- w)^2=2c_1.h\neq 0$, and so reflection through $v'-w$ is a Galois equivariant isometry $v'^\perp/\langle v' \rangle \cong w^\perp/\langle w\rangle$. It can be checked that $v^\perp/\langle v\rangle \xrightarrow{\cdot e^h} v'^\perp /\langle v' \rangle$ is an isometry, and it is Galois equivariant because $h$ is Galois invariant. This completes the proof of $(ii)$.
\end{proof}

\begin{remark}
\label{dim2}
Recall that $H^{i}(\overline{M}, \Q_{\ell})=0$ for odd $i$ \cite[Thm.\ 1(3)]{Ma3}, so the proof of Theorem \ref{mainthm} is complete in the case where $\dim M_1=\dim M_2=2$.
\end{remark}

\section{Reduction to the case of a single surface}
\label{single}

In Section \ref{i=1}, we were able to conclude Theorem \ref{mainthm} holds for $i=2$ by using results about a single K3 surface along with the assumption that $H^2(\overline{S}_1, \Q_\ell)\cong H^2(\overline{S}_2,\Q_\ell)$ as $\mathrm{Gal}(\overline{k}/k)$-representations. By the following proposition, to complete the proof for $i>2$ it is enough to show that $H^i(\overline{M}, \Q_\ell)\cong H^i(\overline{S}^{[n]}, \Q_\ell)$ as $\mathrm{Gal}(\overline{k}/k)$-representations, where $n=\frac{1}{2}\dim M$. 

\begin{prop}
Let $S_1$ and $S_2$ be two K3 surfaces defined over an arbitrary field $k$ such that $H^2(\overline{S}_1, \Q_\ell)\cong H^2(\overline{S}_2,\Q_\ell)$ as $\mathrm{Gal}(\overline{k}/k)$-representations. Then $H^i(\overline{S}_1^{[n]}, \Q_\ell)\cong H^i(\overline{S}_2^{[n]}, \Q_\ell)$ as $\mathrm{Gal}(\overline{k}/k)$-representations for all $i\geq 0$.
\end{prop}

\begin{proof}
For a K3 surface $S$, de Cataldo and Migliorini show in \cite[Thm.\ 6.2.1]{dCM} that the rational Chow motive of $\overline{S}^{[n]}$ is built out of motives of symmetric products $\overline{S}^{(l(\nu))}$ where $\nu$ is a partition of $n$ and $l(\nu)$ is the length of $\nu$. The maps $\overline{S}^{(l(\nu))}\to \overline{S}^{[n]}$ used to give the isomorphism are induced by tautological correspondences defined over the base field, so the decomposition works over any field (see \cite[Rmk.\ 6.2.2]{dCM}). This implies the following $\mathrm{Gal}(\overline{k}/k)$-equivariant isomorphism on the level of cohomology:
$$H^*(\overline{S}^{[n]}, \Q_\ell)\cong \bigoplus_{\nu\in \mathfrak{P}(n)} H^*(\overline{S}^{(l(\nu))}, \Q_\ell)(n-l(\nu)),$$
where $\mathfrak{P}(n)$ is the set of partitions of $n$. Since $H^*(\overline{S}^{(m)}, \Q_\ell)\cong H^*(\overline{S}^m, \Q_\ell)^{\Sigma_m}$ for any $m\geq 1$, where $H^*(\overline{S}^m, \Q_\ell)^{\Sigma_m}$ is the subring of $\Sigma_m$-invariants, the result follows.
\end{proof}

Thus the proof of Theorem \ref{mainthm} will be complete once we know that $H^i(\overline{M}, \Q_\ell)\cong H^i(\overline{S}^{[n]}, \Q_\ell)$ for a given K3 surface.

\begin{remark}
\label{notringiso}
It is interesting to observe that we need not arrive at a ring isomorphism between $H^*(\overline{S}_1^{[n]}, \Q_\ell)$ and $H^*(\overline{S}_2^{[n]}, \Q_\ell)$, and in fact this appears to depend on whether or not the isomorphism $H^2(\overline{S}_1,\Q_\ell)\cong H^2(\overline{S}_2,\Q_\ell)$ as Galois representations agrees with the cohomology ring structures. Indeed, if there is a Galois equivariant ring isomorphism $H^*(\overline{S}_1, \Q_\ell)\cong H^*(\overline{S}_2, \Q_\ell)$, then the intersection forms on the middle cohomology agree and along with Proposition \ref{perps} we get an isometry between their Mukai lattices. Following an argument akin to that given in Proposition \ref{cohomiso} below, this implies the rings $H^*(\overline{S}_1^{[n]}, \Q_\ell)$ and $H^*(\overline{S}_2^{[n]}, \Q_\ell)$ are isomorphic.

If instead the given isomorphism $H^2(\overline{S}_1,\Q_\ell)\cong H^2(\overline{S}_2,\Q_\ell)$ as Galois representations is not an isometry with respect to the intersection pairing, we should not expect $H^*(\overline{S}_1^{[n]}, \Q_\ell)$ and $H^*(\overline{S}_2^{[n]}, \Q_\ell)$ to be isomorphic rings. Suppose there is a ring isomorphism $\psi\colon H^*(\overline{S}_1^{[n]}, \Q_\ell) \xrightarrow{\sim} H^*(\overline{S}_2^{[n]}, \Q_\ell)$, and let $q_i\colon H^2(\overline{S}_i^{[n]}, \Q_\ell)\to \Q_\ell$ for $i=1$ and 2 be the Beauville-Bogomolov form, introduced at the beginning of Section \ref{orthog}. Then for $\alpha \in H^2(\overline{S}_1^{[n]}, \Q_\ell)$,
$$q_1(\alpha)^n=q_2(\psi(\alpha))^n,$$
so that $q_1$ and $q_2$ agree up to an $n^{th}$-root of unity. The only roots of unity in $\Q_\ell$ are the $(\ell-1)$th roots of unity for $\ell$ odd and $\pm 1$ for $\ell=2$, so if we choose $\ell>2$ with $\gcd(n,\ell-1)=1$, this root must be trivial. If $n$ is even, we can only ensure that $\gcd(n,\ell-1)=2$, implying the root is $\pm 1$, but we claim $q_1\not\cong -q_2$.

Consider when $n$ is even and $\ell=3$ so that $\gcd(n,2)=2$. We will show that for the form $q$ on $\Q_3^{23}$ giving $H^2(\overline{S}_1^{[n]}, \Q_3)$, there is no linear isomorphism of $\Q_3^{23}$ taking $q$ to $-q$, and hence the Beauville-Bogomolov forms on $H^2(\overline{S}_1^{[n]}, \Q_3)$ and $H^2(\overline{S}_2^{[n]}, \Q_3)$ cannot differ by a sign. By Propositions \ref{H^2iso} and \ref{perps}, $q$ is given by $(-E_8)^{\oplus 2}\oplus U^{\oplus 3}\oplus \langle 2-2n\rangle$. In the Witt group $W(\Q_3)$, it can be checked that $(-E_8)^{\oplus 2}\oplus U^{\oplus 3}=(E_8)^{\oplus 2}\oplus (-U)^{\oplus 3}=0$, so to see that $q\neq -q\in W(\Q_3)$, we must only check that $\langle 2-2n\rangle \neq \langle -(2-2n)\rangle$ as forms on $\Q_3$. The form $\langle 2-2n\rangle$ is equivalent to $\langle m\rangle$ for $m\in \{-3,-1,1,3\}$, from which it follows that $\langle m \rangle \neq \langle -m \rangle\in W(\Q_3)$ (see \cite[Cor.\ VI.1.6 and Thm.\ VI.2.2]{Lam}).
We conclude that $q_1$ and $q_2$ must agree.

Therefore, again by Propositions \ref{H^2iso} and \ref{perps}, there is a Galois equivariant isometry $H^2(\overline{S}_1, \Q_\ell)\oplus\Q_\ell\cong H^2(\overline{S}_2, \Q_\ell)\oplus\Q_\ell$. As in the proof of Proposition \ref{perps}$(i)$, the reflection that takes the generator of the first $\Q_\ell$ to the generator of the second $\Q_\ell$ restricts to a Galois equivariant isometry\linebreak$H^2(\overline{S}_1, \Q_\ell)\cong H^2(\overline{S}_2, \Q_\ell)$, hence determining the ring structure.
\end{remark}

\section{An isomorphism of the cohomology rings}
\label{theiso}

Let us consider a fixed K3 surface $S$ and a moduli space $M$ of stable sheaves on $S$ with an effective and geometrically primitive Mukai vector $v$. If $v^2=0$, Theorem \ref{mainthm} was proved in Section \ref{i=1} (see Remark \ref{dim2}). Assume now that $v^2>0$. We will continue to use the notation introduced in the proof of Proposition \ref{perps} that $w=(1,0,1-n)\in N(S)$ where $n=\frac{1}{2}(v^2+2)$. We follow \cite[Sec.\ 3.4]{Ma2} to produce an isomorphism between the cohomology rings of $\overline{M}$ and $\overline{S}^{[n]}$ by constructing a class in the middle cohomology of $\overline{M}\times \overline{S}^{[n]}$. This class will depend on the choice of an isometry $g\colon \widetilde{H}(\overline{S}, \Q_\ell) \xrightarrow{\sim} \widetilde{H}(\overline{S}, \Q_\ell)$ such that $g(v)=w$, and we will specifically use the reflection constructed in Proposition \ref{perps}.  

We outline here what Markman does to produce the desired ring isomorphism, where he starts with a complex projective K3 surface and an isometry on $H^*_{sing}(S,\Z)$. By considering an integral isometry, the cohomology class produced is an element of $H^{4n}_{sing}(M\times S^{[n]},\Z)$, and then Markman shows that this class induces a ring isomorphism
$$H^*_{sing}(M,\Z)\xrightarrow{\sim} H^*_{sing}(S^{[n]},\Z).$$
Since we will start with an isometry over $\Q_\ell$, the resulting class, and hence the map on cohomology, will also be defined over $\Q_\ell$. 
We will make a density argument to reduce to Markman's result and conclude that this map on $\Q_\ell$-cohomology is also an isomorphism.

In order to produce a map $H^*(\overline{M},\Q_\ell)\to H^*(\overline{S}^{[n]},\Q_\ell)$, we would like to compose cohomological Fourier-Mukai transforms with the isometry $g$. First, we have the map $H^*(\overline{M},\Q_\ell)\to H^*(\overline{S}, \Q_\ell)$ induced by the class $u_v$ in the cohomology of $\overline{S}\times_{\overline{k}}\overline{M}$, where $u_v$ is the pullback from $S\times_k {M}$ to $\overline{S}\times_{\overline{k}} \overline{{M}}$ of a normalization of $v(\mathcal{U})(\mathrm{td}\,M)^{-1/2}$, defined in \cite[Eq.\ (3.4)]{Ma2}. This is followed by $g\colon \widetilde{H}(\overline{S}, \Q_\ell)\to \widetilde{H}(\overline{S}, \Q_\ell)$, and the last map $H^*(\overline{S}, \Q_\ell)\to H^*(\overline{S}^{[n]},\Q_\ell)$ is induced by the class $u_w$, defined analogously to $u_v$. This composition clearly will not give an isomorphism, but the morphism Markman constructs is induced by a class built out of $u_v$, $g$, and $u_w$, as described below.

For a projective variety $X$, consider the universal polynomial map 
$$l\colon \oplus_i H^{2i}(\overline{X}, \Q_\ell) \to \oplus_i H^{2i}(\overline{X}, \Q_\ell)$$
taking the Chern character of a sheaf to its total Chern class. That is, 
\[l(r + a_1 +a_2 + \cdots)= 1 + a_1 + \left(\frac{1}{2}a_1^2-a_2\right)+\cdots .\]
Let $\pi_{ij}$ be the projection from $\overline{M}\times \overline{S}\times \overline{S}^{[n]}$ onto the product of the $i^{th}$ and $j^{th}$ factors. We define
\begin{equation*}
\label{isomsm}
\gamma_g := c_{2n}\big(l(-\pi_{13*}[\pi_{12}^*((1\otimes g)(u_v))^\vee \pi_{23}^*(u_w)] )\big),
\end{equation*}
so that $\gamma_g \in H^{4n}(\overline{M}\times \overline{S}^{[n]}, \Q_\ell(2n))$, the middle cohomology group. For further discussion on this choice of cohomology class, see \cite[Sec.\ 3.4]{Ma2}.

Now consider the projections from $\overline{M}\times \overline{S}^{[n]}$:
\[\xymatrix{  &  \overline{M}\times \overline{S}^{[n]} \ar[dl]_{q} \ar[dr]^{p}&  \\
			\overline{M} & 			&  \overline{S}^{[n]}. }
\]
We also let $\gamma_g$ denote the induced map $H^*(\overline{M}, \Q_\ell)\to H^*(\overline{S}^{[n]}, \Q_\ell)$ given by 
$$\alpha \mapsto p_*(q^*(\alpha)\cdot\gamma_g).$$

\begin{prop}
\label{cohomiso}
Let $S$ be a K3 surface defined over an arbitrary field $k$ and $v\in N(S)$ an effective and geometrically primitive Mukai vector of length $v^2>0$ with a $v$-generic polarization $H\in \mathrm{NS}(S)$. Let $g\colon \widetilde{H}(\overline{S},\Q_\ell)\to\widetilde{H}(\overline{S},\Q_\ell)$ denote the Galois-equivariant isometry produced in the proof of Proposition \ref{perps}. Then the map $\gamma_g\colon H^*(\overline{M}, \Q_\ell)\to H^*(\overline{S}^{[n]}, \Q_\ell)$ is a Galois equivariant ring isomorphism.
\end{prop}

\begin{remark}
Note that to prove Theorem \ref{mainthm}, $\gamma_g$ just needs to be a Galois-equivariant vector space isomorphism. The fact that $\gamma_g$ is also a ring isomorphism is independently interesting, especially when contrasted with Remark \ref{notringiso}.
\end{remark}

\begin{proof}
We begin by assuming that $k=\C$ and the cohomology is singular cohomology. Let $I:=\mathrm{Isom}(\widetilde{H}(S), v, w)$ be the subvariety of $\A^{24\times 24}_{\Q_\ell}$ consisting of isometries $\widetilde{H}(S) \to \widetilde{H}(S)$ which send $v$ to $w$. Similarly, let $\Hom(H^*(M), H^*(S^{[n]}))$ be the affine variety of graded vector space homomorphisms from $H^*(M)$ to  $H^*(S^{[n]})$. Then we get a map of varieties 
$$\Psi\colon I \to \Hom(H^*(M), H^*(S^{[n]}))$$
sending an isometry $g$ to the map $\gamma_g$ defined above. Consider the subspace $Z$ of $I$ containing all those isometries $g$ such that $\gamma_g$ is a ring homomorphism. We will show that $Z=I$. Observe that $\gamma_g$ being a ring homomorphism is a closed condition so $\Psi(Z) \subset \Hom(H^*(M), H^*(S^{[n]}))$ is closed. Since $Z$ is the preimage under $\Psi$ of a closed subspace, $Z\subset I$ is closed.

Given a $\Z$-point $g\colon \widetilde{H}(S, \Z)\to \widetilde{H}(S, \Z)$ of $I$, by \cite[Thm.\ 3.10]{Ma2} the map $\gamma_g \colon H^*(M,\Z)\to H^*(S^{[n]},\Z)$ is a ring homomorphism, so $Z$ contains all of the $\Z$-points of $I$, $I(\Z)$. By Lemma \ref{BDT} below, we see that the $\Z$-points of $I$ are Zariski dense in $I$, i.e $\overline{I(\Z)}=I$. Since $Z$ is closed, $\overline{I(\Z)}=I\subset Z$. Thus, we conclude that every morphism $\gamma_g$ for $g\in I$ is a ring homomorphism.

Next, we claim that in fact every homomorphism in Im$(\Psi)$ is a ring isomorphism. 
We consider the algebraic map
$$I \to \Hom(H^*(M), H^*(M))$$
sending $g \mapsto \gamma_{g^{-1}} \gamma_g$, where $\gamma_{g^{-1}}$ is defined analogously to $\gamma_g$ for $g^{-1}\in \mathrm{Isom}(\widetilde{H}(S), w, v)$.  Then the subspace 
$$\left\{g : \gamma_{g^{-1}} \gamma_g -\mathrm{Id}=0\right\}\subset I$$
is again closed because it is the preimage of a closed subspace in $\Hom(H^*(M), H^*(M))$. When $g$ is a $\Z$-point of $I$, by \cite[Lem.\ 3.12]{Ma2} we know that $\gamma_{g^{-1}} \gamma_g=\gamma_{g^{-1}g}=\gamma_{\mathrm{Id}}=\mathrm{Id}$. Thus this closed subspace contains all of the $\Z$-points of $I$. Again using Lemma \ref{BDT}, the $\Z$-points of $I$ are Zariski dense in $I$, so we conclude that $\gamma_{g^{-1}} \gamma_g = \mathrm{Id}$ for all $g\in I$. The same argument shows that $\gamma_{g} \gamma_{g^{-1}} = \gamma_{gg^{-1}}=\mathrm{Id}$, and hence every such $\gamma_g$ is an isomorphism. In particular, the isometry $g$ constructed in Proposition \ref{perps} is a $\Q_\ell$-point of $I$ and therefore $\gamma_g$ is a ring isomorphism. Lastly, the comparison theorem for singular and \'etale cohomology gives the ring isomorphism on \'etale cohomology, $\gamma_g\colon H^*(M, \Q_\ell)\xrightarrow{\sim} H^*(S^{[n]}, \Q_\ell)$. 

For $k$ an arbitrary field, we proceed as in the proof of Proposition \ref{H^2iso}, using the Lefschetz principle and the lifting argument for fields of characteristic zero and $p>0$, respectively, to conclude that $\gamma_g$ remains a ring isomorphism.

To see that $\gamma_g$ is Galois equivariant, we observe that both $u_v$ and $u_w$ are Galois invariant, and all of the other operations in the construction of the class $\gamma_g$, including the choice of isometry $g$, are Galois equivariant. Hence the class $\gamma_g$ is invariant under the Galois action, and the resulting morphism is equivariant.
\end{proof}

\begin{lemma}
\label{BDT}
Using the notation introduced in the proof of Proposition \ref{cohomiso}, the $\Z$-points of the variety $\mathrm{Isom}(\widetilde{H}(S), v, w)$ in $\A_{\Q_\ell}^{24\times 24}$ are Zariski dense.
\end{lemma}

\begin{proof}
First we consider $I:=\mathrm{Isom}(\widetilde{H}({S}),v,w)\subset \A_\C^{24\times 24}$. We will show that $\overline{I(\Z)}=I$. Then under a choice of inclusion $\Q_\ell\hookrightarrow \C$, this will give $I(\Q_\ell)\subset \overline{I(\Z)}$, completing the proof.

Consider $I_\R\subset \A_\R^{24\times 24}$, which is a torsor over 
$$\mathrm{Stab}(v)_\R:=\{A\in \mathrm{O}(\widetilde{H}(S)) : Av=v\}\subset \A_\R^{24\times 24},$$ 
and observe that $\mathrm{Stab}(v)_\R$ is isomorphic to $\mathrm{O}(v^\perp)\cong \mathrm{O}(3,20)\subset \A_\R^{23\times 23}$ as group schemes over $\Q$. By \cite[Thm.\ 5.1.11]{Morris}, the $\Z$-points, which we will write as $\mathrm{Stab}(v)_\Z$, form a lattice in $\mathrm{Stab}(v)_\R$. We claim that this lattice is Zariski dense. To see this, we will apply the Borel Density Theorem given in \cite[Cor.\ 4.5.6]{Morris}, which requires a connected, semisimple Lie group $G$ and a lattice $\Gamma$ that projects densely into the maximal compact factor of $G$. Let $G:=\mathrm{Stab}(v)_\R^\circ$, the connected component of the identity, and $\Gamma:=\mathrm{Stab}(v)_\Z\cap G$. Since $\mathrm{Stab}(v)_\R\cong \mathrm{O}(3,20)$, it follows that $G\cong \mathrm{SO}(3,20)^\circ$, which is a simple Lie group by \cite[Ex.\ A2.3(2)]{Morris}. Then by \cite[Ex.\ 4.3.2]{Morris} and \cite[Exercise 4.3\#1]{Morris}, $\Gamma$ projects densely into the maximal compact subgroup of $G$. Thus by \cite[Cor.\ 4.5.6]{Morris}, $\overline{\Gamma}=G$. Therefore ${\mathrm{Stab}(v)_\R^\circ}\subset \overline{\mathrm{Stab}(v)_\Z}$. Recall that $$v^\perp\cong (-E_8)^{\oplus 2}\oplus U^{\oplus 3}\oplus \langle 2-2n\rangle,$$ 
so $\mathrm{Stab}(v)_\Z$ also contains a point of determinant $-1$. The smallest algebraic group containing both $\mathrm{Stab}(v)_\R^\circ$ and this point of determinant $-1$ is $\mathrm{Stab}(v)_\R$, since $\mathrm{O}(3,20)$ is generated by $\mathrm{SO}(3,20)$ and an element of determinant $-1$. Thus in fact $\mathrm{Stab}(v)_\R\subset \overline{\mathrm{Stab}(v)_\Z}$.

Finally, we observe that $\mathrm{Stab}(v)_\R$ is Zariski dense in its complexification $\mathrm{Stab}(v)_\C\subset \A_\C^{24\times 24}$ (see \cite[Rmk.\ 18.1.8(3)]{Morris}), and so $\mathrm{Stab}(v)_\Z$ is Zariski dense in $\mathrm{Stab}(v)_\C\cong \mathrm{O}(23,\C)$. Since $I$ is a torsor over $\mathrm{Stab}(v)_\C$, when we consider $I\subset \A_\C^{24\times 24}$, we see that $\overline{I(\Z)}=I$.  
\end{proof}

\bibliographystyle{plain}
\bibliography{bibliography}

\vspace{.2cm}

\noindent\scriptsize{{\it E-mail address}: \texttt{sarah.frei@rice.edu}}

\vspace{.2cm}

\noindent\scriptsize{\textsc{Department of Mathematics, Rice University, Herman Brown Hall, Houston, TX, 77005, United States}}

\end{document}